\documentclass[11pt]{article}

\usepackage[margin=1in]{geometry}
\usepackage{amsmath,amssymb,amsthm}
\usepackage{algorithm}
\usepackage{algpseudocode}
\usepackage[colorlinks=true,linkcolor=blue,citecolor=blue,urlcolor=blue]{hyperref}
\usepackage{graphicx}
\newtheorem{theorem}{Theorem}
\newtheorem{proposition}{Proposition}
\newtheorem{remark}{Remark}

\newenvironment{classification}
  {\par\medskip\noindent\textbf{Mathematics Subject Classification.}\ }
  {\par}
\newenvironment{keywords}
  {\par\medskip\noindent\textbf{Keywords.}\ }
  {\par\medskip}

\DeclareMathOperator{\prox}{prox}
\newcommand{\norm}[1]{\lVert#1\rVert}

\title{ADMM Fails to Achieve an $O(K^{-1})$ Ergodic KKT Residual Bound}

\author{
Kaihuang Chen
\thanks{Department of Applied Mathematics, The Hong Kong Polytechnic University, Hung Hom, Hong Kong \\ Email: kaihuang.chen@connect.polyu.hk}
\and
Defeng Sun
\thanks{Department of Applied Mathematics, The Hong Kong Polytechnic University, Hung Hom, Hong Kong \\ Email: defeng.sun@polyu.edu.hk}
\and
Yancheng Yuan
\thanks{Department of Applied Mathematics, The Hong Kong Polytechnic University, Hung Hom, Hong Kong \\ Email: yancheng.yuan@polyu.edu.hk}
\and
Guojun Zhang
\thanks{Department of Applied Mathematics, The Hong Kong Polytechnic University, Hung Hom, Hong Kong \\ Email: guojun.zhang@connect.polyu.hk}
\and
Xinyuan Zhao
\thanks{Department of Mathematics, Beijing University of Technology, Beijing, P.R. China \\ Email: xyzhao@bjut.edu.cn}
}

\date{\today}

\begin{document}
\maketitle

\begin{abstract}
The Karush--Kuhn--Tucker (KKT) residual is a fundamental measure of
first-order optimality and, under an error bound condition, is comparable to the distance to the KKT solution set up to constant factors. Despite the $O(K^{-1})$ ergodic rates known for objective
error and feasibility violations, we show that the KKT residual of
classical ADMM cannot, in general, satisfy a uniform $O(K^{-1})$
bound. Specifically, we construct a fixed-dimensional,
horizon-dependent family of two-block convex optimization problems
for which the KKT residual is $\Omega(K^{-1/2})$ at both the last
iterate and the equal-weight ergodic average at the prescribed
horizon $K$. Consequently, a uniform $O(K^{-1})$ KKT residual bound
is impossible for either output.

\end{abstract}



\section{Introduction}
Let $\mathcal Y$, $\mathcal Z$, and $\mathcal X$ be finite-dimensional real
Euclidean spaces. We consider the two-block linearly constrained convex
optimization problem
\begin{equation}
\label{eq:general-model}
    \min_{y\in\mathcal Y,\,z\in\mathcal Z}
    \left\{ f(y)+g(z) \;\middle|\; B_1y+B_2z=c \right\},
\end{equation}
where $f:\mathcal Y\to(-\infty,+\infty]$ and
$g:\mathcal Z\to(-\infty,+\infty]$ are proper closed convex functions,
$B_1:\mathcal Y\to\mathcal X$ and $B_2:\mathcal Z\to\mathcal X$ are
linear operators, and $c\in\mathcal X$. For a penalty parameter $\sigma>0$, the augmented Lagrangian associated
with \eqref{eq:general-model} is
\begin{equation}
\label{eq:augmented-lagrangian}
    \mathcal L_\sigma(y,z;x)
    :=f(y)+g(z)
    +\langle x,B_1y+B_2z-c\rangle
    +\frac{\sigma}{2}\norm{B_1y+B_2z-c}^2.
\end{equation}
The Karush--Kuhn--Tucker (KKT) conditions of \eqref{eq:general-model} are
\begin{equation}
\label{eq:kkt}
    0\in\partial f(y)+B_1^*x,\qquad
    0\in\partial g(z)+B_2^*x,\qquad
    B_1y+B_2z-c=0.
\end{equation}We assume that the KKT solution set of \eqref{eq:general-model} is nonempty and denote the optimal objective value by $p^\star.$
For $w=(y,z,x)$, we measure the violation of \eqref{eq:kkt} by the KKT
residual
\begin{equation}
\label{eq:general-residual}
    \mathcal R(w):=
    \begin{pmatrix}
        y-\prox_f(y-B_1^*x)\\
        z-\prox_g(z-B_2^*x)\\
        c-B_1y-B_2z
    \end{pmatrix}.
\end{equation}
Clearly, $\mathcal R(w)=0$ if and only if $w$ satisfies
\eqref{eq:kkt}.

Given a step-length parameter
$\tau\in\bigl(0,(1+\sqrt5)/2\bigr)$, the alternating direction method of
multipliers (ADMM) \cite{GlowinskiMarroco1975,GabayMercier1976} for
solving \eqref{eq:general-model} generates a sequence
$\{(y^k,z^k,x^k)\}$ according to
\begin{equation}
\label{eq:admm-al}
\left\{
\begin{aligned}
    y^{k+1}
    &=\operatorname*{argmin}_{y\in\mathcal Y}
      \mathcal L_\sigma(y,z^{k};x^{k}),\\
       z^{k+1}
    &=\operatorname*{argmin}_{z\in\mathcal Z}
      \mathcal L_\sigma(y^{k+1},z;x^k),\\
         x^{k+1}
    &=x^k+\tau\sigma(B_1y^{k+1}+B_2z^{k+1}-c),\\
\end{aligned}
\right.
\end{equation}
Significant progress has been made in the convergence-rate analysis of
ADMM for solving \eqref{eq:general-model}. In particular, Monteiro and
Svaiter \cite{monteiro2013iteration} established an $O(k^{-1})$ ergodic
complexity bound for ADMM with $\tau=1$ in terms of an
$\varepsilon$-KKT certificate. Specifically, at an ergodic point
$(\widetilde y_a^k,\widetilde z_a^k,\widetilde x_a^k)$, there exist
$r_{y,a}^k$, $r_{z,a}^k$ and
$\varepsilon_{y,a}^k,\varepsilon_{z,a}^k\geq0$ such that
\begin{equation}
\label{eq:ms-epsilon-kkt}
\begin{aligned}
    &r_{y,a}^k
    \in\partial_{\varepsilon_{y,a}^k}f(\widetilde y_a^k)
       +B_1^*\widetilde x_a^k,\qquad
    r_{z,a}^k
    \in\partial_{\varepsilon_{z,a}^k}g(\widetilde z_a^k)
       +B_2^*\widetilde x_a^k,\\
    &\norm{r_{y,a}^k}+\norm{r_{z,a}^k}
     +\norm{B_1\widetilde y_a^k+B_2\widetilde z_a^k-c}
     =O(k^{-1}),\\
    &\varepsilon_{y,a}^k+\varepsilon_{z,a}^k=O(k^{-1}),
\end{aligned}
\end{equation}
where
\[
    \partial_\varepsilon h(u)
    :=
    \left\{v\ \middle|\
    h(w)\geq h(u)+\langle v,w-u\rangle-\varepsilon
    \quad\forall\,w\right\}
\]
denotes the $\varepsilon$-subdifferential of $h$. Importantly, the
$O(k^{-1})$ approximate KKT bound in \eqref{eq:ms-epsilon-kkt}
directly yields an $O(k^{-1/2})$ bound for the exact KKT residual
\eqref{eq:general-residual}; see
\cite[Lemma~D.1]{ZhangChenYuanZhaoSun2024}.

Davis and Yin \cite{davisyin2017rates} later established, for classical
ADMM with $\tau=1$, the following ergodic and nonergodic rates for the
objective error and primal feasibility:
\begin{equation}
\label{eq:dy-rates}
\begin{aligned}
    \left|f(y_a^k)+g(z_a^k)-p^\star\right|
    +\norm{B_1y_a^k+B_2z_a^k-c}
    &=O(k^{-1}),\\
    \left|f(y^k)+g(z^k)-p^\star\right|
    +\norm{B_1y^k+B_2z^k-c}
    &=o(k^{-1/2}),
\end{aligned}
\end{equation}
where
\[
    y_a^k:=\frac{1}{k}\sum_{i=1}^k y^i,
    \qquad
    z_a^k:=\frac{1}{k}\sum_{i=1}^k z^i.
\]
They further showed, via the Douglas--Rachford--ADMM equivalence, that for every $\alpha>1/2$, there exists an infinite-dimensional instance
for which the nonergodic primal feasibility
residual of ADMM satisfies
\[
    \norm{B_1y^k+B_2z^k-c}
    =\Omega(k^{-\alpha}).
\]

Regarding the KKT residual, Cui et al.~\cite{cui2016majorized} analyzed a majorized proximal ADMM that includes classical ADMM as a special case. For
$0<\tau<(1+\sqrt{5})/2$, they established the best-iterate nonergodic
KKT residual rate
\begin{equation}
\label{eq:cui-nonergodic-rate}
\begin{aligned}
\min_{1\leq i\leq k}\Big\{&
    \operatorname{dist}\!\left(
        0,\partial f(y^{i+1})+B_1^*x^{i+1}\right)
    +\operatorname{dist}\!\left(
        0,\partial g(z^{i+1})+B_2^*x^{i+1}\right)\\
&\qquad
    +\norm{B_1y^{i+1}+B_2z^{i+1}-c}
    \Big\}
    =o(k^{-1/2}).
\end{aligned}
\end{equation}

Although ergodic ADMM enjoys $O(k^{-1})$ rates for objective error and primal feasibility, these measures need not accurately reflect the distance to the KKT solution set. In contrast, an error-bound condition together with the Lipschitz continuity of the residual mapping implies that the KKT residual \eqref{eq:general-residual} is comparable to this distance up to constant factors. It is therefore natural to ask whether the ergodic average of ADMM can also attain an $O(k^{-1})$ rate for the KKT residual.

In this note, we answer this question negatively by establishing an
$\Omega(k^{-1/2})$ finite-horizon lower bound for the ergodic KKT
residual, together with a lower bound of the same order for the last
iterate. Specifically, for each prescribed square horizon $K$, we
construct a consensus instance of \eqref{eq:general-model} in the fixed
space $\mathbb R^2\times\mathbb R$ such that
\begin{equation}
\label{eq:intro-result}
    \norm{\mathcal R_K(w^K)}
    \geq \frac{1}{4\sqrt{2K}},
    \qquad
    \norm{\mathcal R_K(w_a^K)}
    \geq \frac{1}{8\sqrt{2K}}.
\end{equation}
Here, $\mathcal R_K$ denotes the KKT residual mapping of the instance
associated with horizon $K$, and $w_a^K$ denotes the corresponding
equal-weight ergodic average. Since the horizons $K$ are unbounded,
\eqref{eq:intro-result} rules out any uniform $O(K^{-1})$ bound on the
exact KKT residual for either the last iterate or the equal-weight
ergodic average of ADMM over this problem family.

The remainder of the paper is organized as follows. Section~2 constructs
the hard family, derives the corresponding ADMM iterates, and proves
\eqref{eq:intro-result}. Section~3 discusses recent developments in
accelerated ADMM and their applications, while Section~4 concludes the
paper.

\section{A lower bound for the KKT residual of ADMM}
\label{sec:main-result}

In this section, for each prescribed square horizon $K$, we construct a
problem instance in a fixed-dimensional space for which the KKT
residual of classical ADMM admits an $\Omega(K^{-1/2})$ lower bound at
both the last iterate and the equal-weight ergodic average.

\subsection{A lower bound instance}

Fix an integer $q\geq16$ and set
\[
    K:=q^2,
    \qquad
    A_0:=\frac{1}{\sqrt{2}}.
\]
Let
\[
    \mathcal H:=\mathbb R^2\times\mathbb R,
    \qquad
    y=(y_{\rm s},y_{\rm n}),\quad
    z=(z_{\rm s},z_{\rm n}),
\]
where $y_{\rm s},z_{\rm s}\in\mathbb R^2$ and
$y_{\rm n},z_{\rm n}\in\mathbb R$. Let $e_1=(1,0)^\top\in\mathbb R^2$ and define
\[
    U:=\operatorname{span}(e_1),
    \qquad
    V_K:=
    \operatorname{span}
    \begin{pmatrix}
        \cos\theta_K\\
        \sin\theta_K
    \end{pmatrix},
    \qquad
    \theta_K:=\frac{1}{q}.
\]
Also define
\[
    \phi_K:\mathbb R\to\mathbb R,
    \qquad
    \phi_K(t):=\mu_K\sqrt{t^2+\varepsilon_K^2},
\]
with
\[
    \mu_K:=\frac{A_0}{4q},
    \qquad
    \varepsilon_K:=\frac{\mu_K}{10}.
\]
For a closed convex set $C$, let $\delta_C$ denote its indicator function,
and define
\[
    f_K(y):=\delta_U(y_{\rm s}),
    \qquad
    g_K(z):=\delta_{V_K}(z_{\rm s})+\phi_K(z_{\rm n}).
\]
We consider the consensus problem
\begin{equation}
\label{eq:hard-problem}
    \mathcal P_K:\qquad
    \min_{y,z\in\mathcal H}
    \left\{
        f_K(y)+g_K(z)
        \;\middle|\;
        y-z=0
    \right\}.
\end{equation}

Since $U\cap V_K=\{0\}$ and $\phi_K$ is uniquely minimized at $0$,
the unique primal solution is $y^\star=z^\star=0$. The KKT conditions further
imply
\[
    x_{\rm n}^\star=0,
    \qquad
    x_{\rm s}^\star\in U^\perp\cap V_K^\perp=\{0\}.
\]
Hence, $\mathcal P_K$ has the unique KKT point
\[
    w^\star:=(y^\star,z^\star,x^\star)=(0,0,0).
\]
\subsection{ADMM and its two outputs}

Problem \eqref{eq:hard-problem} is a special case of
\eqref{eq:general-model} with $B_1=I$, $B_2=-I$, and $c=0$. For
$\sigma=1$, the corresponding augmented Lagrangian is
\begin{equation}
\label{eq:hard-augmented-lagrangian}
    \mathcal L_{1,K}(y,z;x)
    :=
    f_K(y)+g_K(z)
    +\langle x,y-z\rangle
    +\frac12\norm{y-z}^2.
\end{equation}
The KKT residual mapping \eqref{eq:general-residual} specializes to
\begin{equation}
\label{eq:consensus-residual}
    \mathcal R_K(y,z,x)
    :=
    \begin{pmatrix}
        y-\prox_{f_K}(y-x)\\
        z-\prox_{g_K}(z+x)\\
        z-y
    \end{pmatrix}.
\end{equation}

We apply ADMM with penalty parameter $\sigma=1$ and step-length
$\tau=1$, initialized at
\begin{equation}
\label{eq:initial-point}
    y_{\rm s}^0=z_{\rm s}^0=0,
    \quad
    x_{\rm s}^0=A_0e_1,
    \qquad
    y_{\rm n}^0=A_0,
    \quad
    z_{\rm n}^0=x_{\rm n}^0=0.
\end{equation}
Since $w^\star=0$ and $A_0=1/\sqrt{2}$, the initial distance to the KKT
point is
\begin{equation}
\label{eq:R0}
    R_0
    :=
    \norm{w^0-w^\star}
    =
    \sqrt{2}\,A_0
    =
    1.
\end{equation}

For convenience, we cyclically shift the starting point of the ADMM update cycle \eqref{eq:admm-al}, without changing any of its subproblems, and write the equivalent update order as $z\to x\to y$. The corresponding iteration, together with its last-iterate and ergodic
outputs, is given in Algorithm~\ref{alg:admm}.

\begin{algorithm}[H]
\caption{ADMM with $\sigma=\tau=1$ on $\mathcal P_K$}
\label{alg:admm}
\begin{algorithmic}[1]
\Require problem $\mathcal P_K$, initial point
    $w^0=(y^0,z^0,x^0)$ from \eqref{eq:initial-point}, and horizon
    $K=q^2$
\For{$k=0,\ldots,K-1$}
    \State $\displaystyle
        z^{k+1}=
        \operatorname*{argmin}_{z\in\mathcal H}
        \mathcal L_{1,K}(y^k,z;x^k)$
    \State $\displaystyle
        x^{k+1}= x^k+y^k-z^{k+1}$
    \State $\displaystyle
        y^{k+1}=
        \operatorname*{argmin}_{y\in\mathcal H}
        \mathcal L_{1,K}(y,z^{k+1};x^{k+1})$
    \State $w^{k+1}=(y^{k+1},z^{k+1},x^{k+1})$
\EndFor
\State \Return $w^K$ and
        $\displaystyle w_a^K=\frac{1}{K}\sum_{j=1}^K w^j$
\end{algorithmic}
\end{algorithm}

\subsection{Complexity lower bound}

We now establish the lower bounds for the two outputs of
Algorithm~\ref{alg:admm}.

\begin{theorem}
\label{thm:lower-bound}
For every integer $q\geq16$, with $K=q^2$, there exists a problem
instance $\mathcal P_K$ such that the last iterate $w^K$ and the
equal-weight ergodic average
$w_a^K:=K^{-1}\sum_{j=1}^K w^j$ generated by Algorithm~\ref{alg:admm} satisfy
\begin{equation}
\label{eq:main-lower-bounds}
\begin{aligned}
    \norm{\mathcal R_K(w^K)}
    &\geq
    \frac{R_0}{4\sqrt{2K}},\\
    \norm{\mathcal R_K(w_a^K)}
    &\geq
    \frac{R_0}{8\sqrt{2K}}.
\end{aligned}
\end{equation}
Consequently, neither output admits a uniform $O(R_0/K)$ worst-case
bound for the exact KKT residual over this problem family.
\end{theorem}

\begin{proof}
The two components of the product space $\mathcal H$ yield the two
lower bounds separately: the subspace component gives the lower bound
for the last iterate, while the scalar component gives the lower bound
for the ergodic average.

\paragraph{Last iterate.}
Let $P_U$ and $P_{V_K}$ denote the orthogonal projections onto $U$ and
$V_K$, respectively, and define
\[
    p_{\rm s}^k:=y_{\rm s}^k+x_{\rm s}^k.
\]
The subspace component of Algorithm~\ref{alg:admm} satisfies
\[
\begin{aligned}
    z_{\rm s}^{k+1}
    &=P_{V_K}p_{\rm s}^k,\\
    x_{\rm s}^{k+1}
    &=(I-P_{V_K})p_{\rm s}^k,\\
    y_{\rm s}^{k+1}
    &=P_U(2P_{V_K}-I)p_{\rm s}^k.
\end{aligned}
\]
Hence
\begin{equation}
\label{eq:subspace-dynamics}
\begin{aligned}
    p_{\rm s}^{k+1}
    &=
    \bigl(P_U(2P_{V_K}-I)+I-P_{V_K}\bigr)p_{\rm s}^k=
    \cos\theta_K\,R_{-\theta_K}p_{\rm s}^k,
\end{aligned}
\end{equation}
where
\[
    R_{-\theta_K}
    :=
    \begin{pmatrix}
        \cos\theta_K & \sin\theta_K\\
        -\sin\theta_K & \cos\theta_K
    \end{pmatrix}
\]
is the rotation matrix through the angle $-\theta_K$.
Since $R_{-\theta_K}$ is orthogonal and
$p_{\rm s}^0=A_0e_1$, \eqref{eq:subspace-dynamics} gives
\[
    \norm{p_{\rm s}^k}
    =
    A_0\cos^k\theta_K.
\]
Moreover, a direct calculation yields
\[
    \norm{z_{\rm s}^{k+1}-y_{\rm s}^{k+1}}
    =
    \sin\theta_K\,\norm{p_{\rm s}^k}.
\]
Therefore, using the feasibility component of
\eqref{eq:consensus-residual},
\[
\begin{aligned}
    \norm{\mathcal R_K(w^K)}
    &\geq
    \norm{z_{\rm s}^{K}-y_{\rm s}^{K}}=
    A_0\sin\theta_K\cos^{K-1}\theta_K=
    A_0\sin(1/q)\cos^{K-1}(1/q),
\end{aligned}
\]
where $\theta_K=1/q$. We next bound the two factors separately. Since
$\sin t\geq t-t^3/6$ for $t\geq0$,
\[
    q\sin(1/q)
    \geq
    1-\frac{1}{6q^2}
    \geq
    \frac12.
\]
Also, since $\cos t\geq1-t^2/2$,
\[
    \cos^{K-1}(1/q)
    =
    \cos^{q^2-1}(1/q)
    \geq
    \left(1-\frac{1}{2q^2}\right)^{q^2-1}.
\]
By Bernoulli's inequality,
\[
    (1+t)^m\geq 1+mt
    \qquad
    \text{for }t\geq-1,\; m\in\mathbb N,
\]
with
\[
    t=-\frac{1}{2q^2},
    \qquad
    m=q^2-1,
\]
we obtain
\[
\begin{aligned}
    \left(1-\frac{1}{2q^2}\right)^{q^2-1}
    &\geq
    1-\frac{q^2-1}{2q^2}
    =
    \frac12+\frac{1}{2q^2}
    \geq
    \frac12.
\end{aligned}
\]
Consequently,
\begin{equation}
\label{eq:last-lower-bound}
    \norm{\mathcal R_K(w^K)}
    \geq
    \frac{A_0}{4q}
    =
    \frac{1}{4\sqrt{2K}}
    =
    \frac{R_0}{4\sqrt{2K}},
\end{equation}
where $K=q^2$, $A_0=1/\sqrt{2}$, and $R_0=1$.

\paragraph{Equal-weight ergodic average.}
Let
\[
    J_K:=\prox_{\phi_K},
    \qquad
    p_{\rm n}^k:=y_{\rm n}^k+x_{\rm n}^k.
\]
Since $p_{\rm n}^0=A_0$, the scalar component of
Algorithm~\ref{alg:admm} satisfies
\begin{equation}
\label{eq:scalar-dynamics}
\begin{aligned}
    p_{\rm n}^{k+1}
    &=J_Kp_{\rm n}^k,\\
    z_{\rm n}^{k+1}
    &=p_{\rm n}^{k+1},\\
    x_{\rm n}^{k+1}
    &=p_{\rm n}^k-p_{\rm n}^{k+1},\\
    y_{\rm n}^{k+1}
    &=2p_{\rm n}^{k+1}-p_{\rm n}^k.
\end{aligned}
\end{equation}
Since $\phi_K$ is even and convex, its proximal mapping preserves
nonnegativity. Hence $p_{\rm n}^j\geq0$ for all $j$. Moreover, the
proximal optimality condition for $p_{\rm n}^j=J_Kp_{\rm n}^{j-1}$ gives
\[
    p_{\rm n}^{j-1}-p_{\rm n}^{j}
    =
    \mu_K
    \frac{p_{\rm n}^{j}}
         {\sqrt{(p_{\rm n}^{j})^2+\varepsilon_K^2}}
    <\mu_K.
\]
Therefore,
\[
    p_{\rm n}^{j}
    \geq
    A_0-j\mu_K
    \geq
    \frac{3A_0}{4},
    \qquad
    1\leq j\leq q,
\]
where we used $\mu_K=A_0/(4q)$.

Using the first $q$ terms in the ergodic average and recalling that
$K=q^2$, we obtain
\[
\begin{aligned}
    z_{a,{\rm n}}^K
    &=
    \frac1K\sum_{j=1}^K p_{\rm n}^j\geq
    \frac1K\sum_{j=1}^q p_{\rm n}^j
    \geq
    \frac{3A_0q}{4K}
    =
    3\mu_K.
\end{aligned}
\]
On the other hand, by telescoping \eqref{eq:scalar-dynamics},
\[
\begin{aligned}
    x_{a,{\rm n}}^K
    &=
    \frac1K\sum_{j=1}^K
    (p_{\rm n}^{j-1}-p_{\rm n}^{j})=
    \frac{A_0-p_{\rm n}^K}{K}\leq
    \frac{A_0}{q^2}
    =
    \frac{4\mu_K}{q}
    \leq
    \frac{\mu_K}{4},
\end{aligned}
\]
where the last inequality follows from $q\geq16$. Thus
\begin{equation}
\label{eq:average-components}
    z_{a,{\rm n}}^K\geq3\mu_K,
    \qquad
    0\leq x_{a,{\rm n}}^K\leq\frac{\mu_K}{4}.
\end{equation}
By the characterization of the proximal mapping and the differentiability of $\phi_K$, we have
\[
\begin{aligned}
    &z_{a,{\rm n}}^K+x_{a,{\rm n}}^K
    -
    J_K\!\left(
        z_{a,{\rm n}}^K+x_{a,{\rm n}}^K
    \right)=
    \mu_K
    \frac{
        J_K\!\left(
            z_{a,{\rm n}}^K+x_{a,{\rm n}}^K
        \right)
    }{
        \sqrt{
            J_K\!\left(
                z_{a,{\rm n}}^K+x_{a,{\rm n}}^K
            \right)^2
            +\varepsilon_K^2
        }
    }.
\end{aligned}
\]
In particular,
\[
    0
    \leq
    z_{a,{\rm n}}^K+x_{a,{\rm n}}^K
    -
    J_K\!\left(
        z_{a,{\rm n}}^K+x_{a,{\rm n}}^K
    \right)
    \leq
    \mu_K.
\]
Hence, by \eqref{eq:average-components},
\[
\begin{aligned}
    J_K\!\left(
        z_{a,{\rm n}}^K+x_{a,{\rm n}}^K
    \right)
    &\geq
    z_{a,{\rm n}}^K-\mu_K
    \geq
    2\mu_K.
\end{aligned}
\]
Since $\varepsilon_K=\mu_K/10$, it follows that
\[
\begin{aligned}
    &z_{a,{\rm n}}^K+x_{a,{\rm n}}^K
    -
    J_K\!\left(
        z_{a,{\rm n}}^K+x_{a,{\rm n}}^K
    \right)\geq
    \mu_K
    \frac{2\mu_K}
         {\sqrt{4\mu_K^2+\mu_K^2/100}}
    \geq
    \frac{3\mu_K}{4}.
\end{aligned}
\]
Therefore, using the second component of
\eqref{eq:consensus-residual},
\[
\begin{aligned}
    \norm{\mathcal R_K(w_a^K)}
    &\geq
    z_{a,{\rm n}}^K
    -
    J_K\!\left(
        z_{a,{\rm n}}^K+x_{a,{\rm n}}^K
    \right)\\
    &=
    z_{a,{\rm n}}^K+x_{a,{\rm n}}^K
    -
    J_K\!\left(
        z_{a,{\rm n}}^K+x_{a,{\rm n}}^K
    \right)
    -
    x_{a,{\rm n}}^K\\
    &\geq
    \frac{3\mu_K}{4}
    -
    \frac{\mu_K}{4}=
    \frac{\mu_K}{2}=
    \frac{A_0}{8q}
    =
    \frac{1}{8\sqrt{2K}}
    =
    \frac{R_0}{8\sqrt{2K}}.
\end{aligned}
\]
Together with \eqref{eq:last-lower-bound}, this proves
\eqref{eq:main-lower-bounds}.

Finally, suppose that there existed a constant $C>0$, independent of
$K$, such that
\[
    \norm{\mathcal R_K(w^K)}
    \leq
    \frac{CR_0}{K}
\]
for every member of the family. Combining this inequality with the
first bound in \eqref{eq:main-lower-bounds} gives
\[
    \sqrt{K}\leq4\sqrt{2}\,C
\]
for every square $K=q^2$ with $q\geq16$, which is impossible because
these horizons are unbounded. The same argument, using the second
bound in \eqref{eq:main-lower-bounds}, rules out a uniform
$O(R_0/K)$ bound for the equal-weight ergodic average.
\end{proof}


\begin{remark}[Scope of the lower bound]
The instance $\mathcal P_K$ depends on the prescribed horizon $K$.
Thus, Theorem~\ref{thm:lower-bound} establishes the finite-horizon
worst-case statement
\[
    \forall\, K=q^2,\quad q\geq16,\qquad
    \exists\,\mathcal P_K
\]
such that the KKT residual at iteration $K$ is
$\Omega(K^{-1/2})$. It does not assert the existence of a single
fixed finite-dimensional instance whose ADMM residual remains
$\Omega(k^{-1/2})$ for all $k$.

\end{remark}

\section{Accelerated ADMM: complexity and solver developments}
\subsection{Accelerated KKT residual complexity}
Recent progress shows that acceleration can improve the nonergodic KKT
residual complexity of ADMM-type methods to $O(k^{-1})$, whereas such
a uniform rate cannot hold for classical ADMM in view of the lower bound established above. In particular, for the general convex optimization problem
\eqref{eq:general-model}, Zhang, Yuan, and Sun
\cite{zhang2022efficient} established a (nonergodic) $O(k^{-1})$
KKT residual rate under the Halpern--Peaceman--Rachford (HPR)
framework without semi-proximal terms. More recently, Sun et al.
\cite{SunYuanZhangZhao2025} developed an accelerated semi-proximal ADMM
framework through a degenerate proximal point method reformulation,
establishing nonergodic $O(k^{-1})$ and $o(k^{-1})$ KKT residual
complexity bounds. To illustrate the resulting improvement over classical ADMM, Algorithm~\ref{alg:halpern-admm} presents a Halpern-accelerated ADMM for \eqref{eq:general-model}.

\begin{algorithm}[H]
\caption{Halpern-accelerated ADMM for \eqref{eq:general-model}}
\label{alg:halpern-admm}
\begin{algorithmic}[1]
\Require penalty parameter $\sigma>0$, relaxation parameter
    $\rho\in(0,2]$, and initial point
    $w^0=(y^0,z^0,x^0)\in
    \operatorname{dom}f\times\operatorname{dom}g\times\mathcal X$
\For{$k=0,1,\ldots$}
    \State $\displaystyle
        \bar z^{k+1}=
        \operatorname*{argmin}_{z\in\mathcal Z}
        \mathcal L_\sigma(y^k,z;x^k)$
    \State $\displaystyle
        \bar x^{k+1}=
        x^k+\sigma(B_1y^k+B_2\bar z^{k+1}-c)$
    \State $\displaystyle
        \bar y^{k+1}=
        \operatorname*{argmin}_{y\in\mathcal Y}
        \mathcal L_\sigma(y,\bar z^{k+1};\bar x^{k+1})$
    \State $\displaystyle
        \bar w^{k+1}=(\bar y^{k+1},\bar z^{k+1},\bar x^{k+1})$
    \State $\displaystyle
        \widehat w^{k+1}
        =(1-\rho)w^k+\rho\bar w^{k+1}$
    \State $\displaystyle
        w^{k+1}
        =
        \frac{1}{k+2}w^0
        +\frac{k+1}{k+2}\widehat w^{k+1}$
\EndFor
\end{algorithmic}
\end{algorithm}

Here $\rho\in(0,2]$ is the relaxation parameter; the choice $\rho=2$
corresponds to the HPR method. The KKT residual complexity of
Algorithm~\ref{alg:halpern-admm}, established by Sun et al.\
\cite[Theorem~3.7]{SunYuanZhangZhao2025}, can be summarized as follows.

\begin{theorem}
\label{thm:halpern-admm-complexity}
Assume that the objective function of each ADMM subproblem in Algorithm~\ref{alg:halpern-admm} is strongly
convex and let $w^\star=(y^\star,z^\star,x^\star)$ be a KKT point.
Define $\mathcal M_0$ and its induced seminorm by
\begin{equation}
\label{eq:halpern-admm-metric}
    \mathcal M_0
    :=
    \begin{pmatrix}
        \sigma B_1^*B_1 & 0 & B_1^*\\
        0 & 0 & 0\\
        B_1 & 0 & \sigma^{-1}I_{\mathcal X}
    \end{pmatrix},
    \qquad
    \norm{u}_{\mathcal M_0}
    :=\sqrt{\langle u,\mathcal M_0u\rangle}.
\end{equation}
Then, for every $k\geq0$,
\begin{equation}
\label{eq:halpern-admm-kkt-complexity}
    \norm{\mathcal R(\bar w^{k+1})}
    \leq
    \frac{2(\sigma\norm{B_1^*}+1)
          \norm{w^0-w^\star}_{\mathcal M_0}}
         {\rho\sqrt{\sigma}\,(k+1)}.
\end{equation}
\end{theorem}

\begin{remark}
    For the hard instance $\mathcal P_K$ in Section~2, we have
$\sigma=1$, $B_1=I$, $B_2=-I$, and $w^\star=0$, so the strong-convexity
condition in Theorem~\ref{thm:halpern-admm-complexity} holds.
Moreover, by \eqref{eq:initial-point} and $A_0=1/\sqrt2$,
\[
    \norm{w^0-w^\star}_{\mathcal M_0}^2
    =
    \norm{y^0+x^0}^2
    =
    \norm{x_{\rm s}^0}^2+\lvert y_{\rm n}^0\rvert^2
    =
    2A_0^2
    =
    1.
\]
Since $\norm{B_1^*}=1$, Theorem~\ref{thm:halpern-admm-complexity}
therefore yields
\begin{equation}
\label{eq:halpern-hard-instance-bound}
    \norm{\mathcal R_K(\bar w^{k+1})}
    \leq
    \frac{4}{\rho(k+1)}.
\end{equation}
In particular, after $K$ ADMM sweeps, the output $\bar w^K$ satisfies
\begin{equation}
\label{eq:hpr-hard-instance-bound}
    \norm{\mathcal R_K(\bar w^K)}
    \leq
    \frac{4}{\rho K}.
\end{equation}
Thus, on the same instance for which classical ADMM has the
$\Omega(K^{-1/2})$ lower bounds in
\eqref{eq:main-lower-bounds}, the Halpern-accelerated method has the
explicit $4/(\rho K)$ KKT residual upper bound.

Figure~\ref{fig:fixed-hard-instance} illustrates this distinction for
one fixed instance $\mathcal P_K$ with $q=32$ and prescribed horizon
$K=1024$. The residual is plotted against the iteration index $k$,
while the lower bound in Theorem~\ref{thm:lower-bound} concerns its
value specifically at the prescribed horizon $k=K$.

\begin{figure}[H]
    \centering
    \includegraphics[width=0.8\linewidth]
    {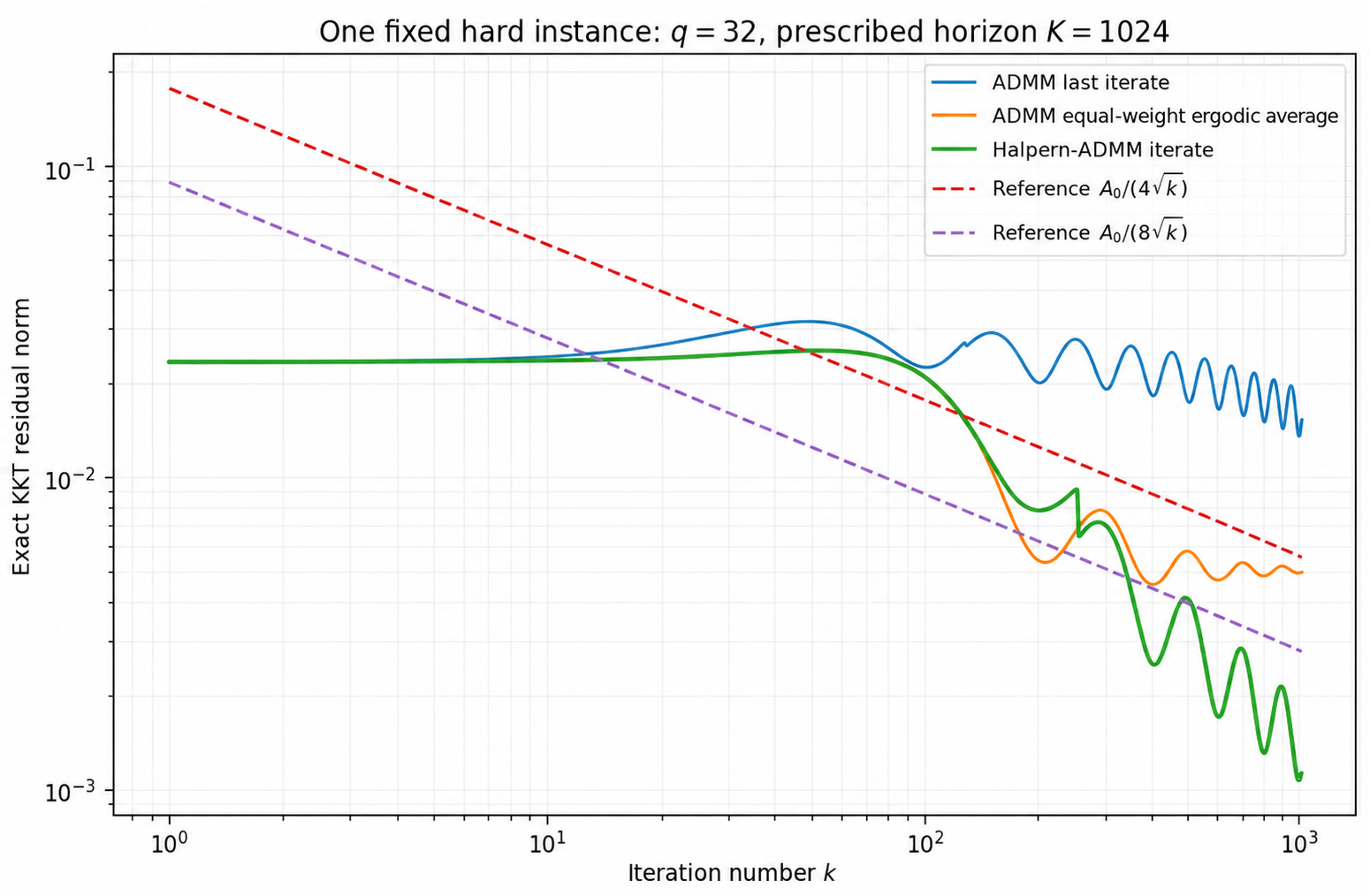}
    \caption{KKT residual along the ADMM iterates for the fixed hard
    instance $\mathcal P_K$ with $q=32$ and prescribed horizon
    $K=1024$, where the ergodic output at iteration $k$ is
    $k^{-1}\sum_{j=1}^k w^j$ and the Halpern--ADMM curve uses
    $\rho=1$. The theorem provides an
    $\Omega(K^{-1/2})$ lower bound at the prescribed iteration $k=K$.}
\label{fig:fixed-hard-instance}
\end{figure}
\end{remark}

\subsection{Comparison with ergodic averaging for LP}
We next illustrate the advantage of acceleration over standard ergodic
averaging for linear programming (LP), a special case of
\eqref{eq:general-model}. We construct a fixed-dimensional,
horizon-dependent LP family for which, at the prescribed horizon $K$,
the KKT residual and the distance to the KKT solution set are
$\Theta(K^{-1/2})$ for the equal-weight ergodic average, but
$\Theta(K^{-1})$ for the Halpern output.

Fix $q\in 8\mathbb N$ and set
\begin{equation}
\label{eq:lp-hard-parameters}
    K:=q^2,\qquad
    a_0:=\frac1{\sqrt2},\qquad
    \gamma_K:=\frac{a_0}{q},\qquad
    \theta_K:=\frac{\pi}{4q},\qquad
    a_K:=\sqrt2\sin\theta_K.
\end{equation}
Define
\begin{equation}
\label{eq:lp-hard-data}
    c_K:=
    \begin{pmatrix}\gamma_K\\\gamma_K\\0\end{pmatrix},
    \qquad
    A_K:=
    \begin{pmatrix}
        1&-1&0\\
        0&0&a_K
    \end{pmatrix},
    \qquad
    b_K:=
    \begin{pmatrix}0\\a_K\end{pmatrix}.
\end{equation}
Consider the primal--dual LP pair
\begin{align}
\label{eq:lp-hard-primal}
    \min_{x\in\mathbb R_+^3}\quad
    &\gamma_K(x_1+x_2)
    &&\text{s.t.}\quad A_Kx=b_K,\\
\label{eq:lp-hard-dual}
    \min_{y\in\mathbb R^2,\ z\in\mathbb R^3}\quad
    &\delta_{\mathbb R_+^3}(z)-a_Ky_2
    &&\text{s.t.}\quad A_K^*y+z=c_K.
\end{align}
The primal problem has the unique solution
$x^\star=(0,0,1)$, and the KKT solution set is
\begin{equation}
\label{eq:lp-kkt-solution-set}
    \mathcal W_{D,K}^\star
    =
    \left\{
    \left(
        (y_1,0),
        (\gamma_K-y_1,\gamma_K+y_1,0),
        (0,0,1)
    \right)
    \ \middle|\ \lvert y_1\rvert\leq\gamma_K
    \right\}.
\end{equation}
For $\sigma=1$, define
\begin{equation}
\label{eq:lp-dual-augmented-lagrangian}
    \mathcal L_{1,\mathrm{LP}}^D(y,z;x)
    :=
    \delta_{\mathbb R_+^3}(z)-a_Ky_2
    +\langle x,A_K^*y+z-c_K\rangle
    +\frac12\norm{A_K^*y+z-c_K}^2.
\end{equation}
We fix
\(
    \lambda_A:=2.
\)
Since
\[
    A_KA_K^*=\operatorname{Diag}(2,a_K^2),
    \qquad
    \lambda_AI-A_KA_K^*
    =
    \operatorname{Diag}(0,2-a_K^2)\succeq0,
\]
this choice is valid for every member of the family. We apply the linearized ADMM, which is equivalent to a PDHG scheme \cite{fazel2013hankel,esser2010general,chambolle2011first}, with update order $z\to x\to y$, as given in Algorithm~\ref{alg:lp-ladmm}.

\begin{algorithm}[H]
\caption{Linearized ADMM for \eqref{eq:lp-hard-dual} with fixed $\lambda_A=2$}
\label{alg:lp-ladmm}
\begin{algorithmic}[1]
\State Choose the initial point
$\displaystyle
w^0=(y^0,z^0,x^0)
=\left((0,0),c_K,(a_0,a_0,1+q^{-1})\right)
$ and set $\lambda_A=2$.
\For{$k=0,\ldots,K-1$}
    \State $\displaystyle
        z^{k+1}
        =
        \operatorname*{argmin}_{z\in\mathbb R^3}\,
        \mathcal L_{1,\mathrm{LP}}^D(y^k,z;x^k)$
    \State $\displaystyle
        x^{k+1}
        =
        x^k+A_K^*y^k+z^{k+1}-c_K$
    \State $\displaystyle
        y^{k+1}
        =
        \operatorname*{argmin}_{y\in\mathbb R^2}
        \left\{
        \mathcal L_{1,\mathrm{LP}}^D
        (y,z^{k+1};x^{k+1})
        +\frac12
        \norm{y-y^k}^2_{\lambda_A I-A_KA_K^*}
        \right\}$
\EndFor
\State \Return
    $\displaystyle
    w_a^K=\frac1K\sum_{j=1}^K w^j$,
    where $\displaystyle
    w^j=(y^j,z^j,x^j)
    \in\mathbb R^2\times\mathbb R^3\times\mathbb R^3$.
\end{algorithmic}
\end{algorithm}

For comparison, we apply Halpern acceleration to the same linearized ADMM sweep \cite{SunYuanZhangZhao2025}.

\begin{algorithm}[H]
\caption{Halpern-accelerated linearized ADMM for
\eqref{eq:lp-hard-dual} with fixed $\lambda_A=2$ and $\rho=1$}
\label{alg:lp-halpern-ladmm}
\begin{algorithmic}[1]
\State Choose the initial point
$\displaystyle
w^0=(y^0,z^0,x^0)
=
\left((0,0),c_K,(a_0,a_0,1+q^{-1})\right)
$
and set $\lambda_A=2$.
\For{$k=0,\ldots,K-1$}
    \State $\displaystyle
        \bar z^{k+1}
        =
        \operatorname*{argmin}_{z\in\mathbb R^3}\,
        \mathcal L_{1,\mathrm{LP}}^D(y^k,z;x^k)$
    \State $\displaystyle
        \bar x^{k+1}
        =
        x^k+A_K^*y^k+\bar z^{k+1}-c_K$
    \State $\displaystyle
        \bar y^{k+1}
        =
        \operatorname*{argmin}_{y\in\mathbb R^2}
        \left\{
        \mathcal L_{1,\mathrm{LP}}^D
        (y,\bar z^{k+1};\bar x^{k+1})
        +\frac12
        \norm{y-y^k}^2_{\lambda_A I-A_KA_K^*}
        \right\}$
    \State $\displaystyle
        \bar w^{k+1}
        =
        (\bar y^{k+1},\bar z^{k+1},\bar x^{k+1})$
    \State $\displaystyle
        w^{k+1}
        =
        \frac1{k+2}w^0
        +
        \frac{k+1}{k+2}\bar w^{k+1}$
\EndFor
\State \Return $\bar w^K$.
\end{algorithmic}
\end{algorithm}

For \eqref{eq:lp-hard-primal}--\eqref{eq:lp-hard-dual}, define the
primal objective and dual objective by
\[
    p_K(x)
    :=
    \gamma_K(x_1+x_2)+\delta_{\mathbb R_+^3}(x),
    \qquad
    d_K(y,z)
    :=
    a_Ky_2-\delta_{\mathbb R_+^3}(z).
\]
Thus \eqref{eq:lp-hard-dual} minimizes $-d_K$, and the primal--dual
gap is $p_K(x)-d_K(y,z)$.
We define the KKT residual for \eqref{eq:lp-hard-dual} by
\begin{equation}
\label{eq:lp-dual-kkt-residual}
    \mathcal R_{D,K}(y,z,x)
    :=
    \begin{pmatrix}
        A_Kx-b_K\\
        z-\Pi_{\mathbb R_+^3}(z-x)\\
        c_K-A_K^*y-z
    \end{pmatrix}.
\end{equation}
\begin{proposition}
\label{prop:lp-ergodic-halpern}
Let $q\in 8\mathbb N$, $K=q^2$, and consider the LP
\eqref{eq:lp-hard-primal}--\eqref{eq:lp-hard-dual}.
The equal-weight ergodic output $w_a^K$ of
Algorithm~\ref{alg:lp-ladmm} satisfies
\begin{equation}
\label{eq:lp-ergodic-rates}
\begin{aligned}
    \frac{7}{16\sqrt K}
    &\leq
    \norm{\mathcal R_{D,K}(w_a^K)}
    \leq
    \frac{1}{\sqrt K},\\
    \frac{7}{16\sqrt K}
    &\leq
    \operatorname{dist}(w_a^K,\mathcal W_{D,K}^\star)
    \leq
    \frac{1}{\sqrt K},\\
    0
    &\leq
    p_K(x_a^K)-d_K(y_a^K,z_a^K)
    =
    \frac{q-1}{2qK}
    +\frac{1-\cos^K\theta_K}{qK}
    \leq
    \frac{5}{8K}.
\end{aligned}
\end{equation}
In contrast, the output $\bar w^K$ of the Halpern-accelerated
Algorithm~\ref{alg:lp-halpern-ladmm} satisfies
\begin{equation}
\label{eq:lp-halpern-rate}
\begin{aligned}
    \frac{1}{K}
    &\leq
    \norm{\mathcal R_{D,K}(\bar w^K)}
    \leq
    \frac{\sqrt{1+3/K}}{K},\\
    \frac{1}{K}
    &\leq
    \operatorname{dist}(\bar w^K,\mathcal W_{D,K}^\star)
    \leq
    \frac{\sqrt{3+1/K}}{K},\\
    0
    &\leq
    p_K(\bar x^K)-d_K(\bar y^K,\bar z^K)
    =
    \frac{1-\cos^K\theta_K}{qK}
    \leq
    \frac{1}{K^{3/2}}.
\end{aligned}
\end{equation}
Consequently, the KKT residual and solution-set distance are
$\Theta(K^{-1/2})$ for the equal-weight ergodic output, but
$\Theta(K^{-1})$ for the Halpern output.
\end{proposition}
\begin{proof}
We first analyze Algorithm~\ref{alg:lp-ladmm}. By symmetry of the
first two coordinates and $y_1^0=0$, a direct induction gives
\begin{equation}
\label{eq:lp-ladmm-first-block-orbit}
    y_1^k=0,\qquad
    (x_1^k,x_2^k)
    =
    a_0\left(1-\frac{k}{q}\right)_+
    \begin{pmatrix}1\\1\end{pmatrix},
\end{equation}
and
\[
    (z_1^k,z_2^k)=
    \begin{cases}
        (\gamma_K,\gamma_K), & k=0\ \text{or}\ k\geq q+1,\\
        (0,0),               & 1\leq k\leq q.
    \end{cases}
\]

For the third coordinate, set
\[
    u^k:=x_3^k-1,
    \qquad
    v^k:=a_Ky_2^k.
\]
The $z$-update gives
\[
    z_3^{k+1}
    =
    \max\{-1-u^k-v^k,0\}.
\]
On the branch $z_3^{k+1}=0$, the $x$-update and the optimality
condition for the $y_2$-subproblem yield
\[
    u^{k+1}=u^k+v^k,
    \qquad
    v^{k+1}
    =
    (1-\sin^2\theta_K)v^k
    -\sin^2\theta_K\,u^{k+1},
\]
where we used
$a_K^2/\lambda_A=\sin^2\theta_K$. Hence
\begin{equation}
\label{eq:lp-ladmm-third-block-recurrence}
    \begin{pmatrix}u^{k+1}\\v^{k+1}\end{pmatrix}
    =
    \begin{pmatrix}
        1&1\\
        -\sin^2\theta_K&1-2\sin^2\theta_K
    \end{pmatrix}
    \begin{pmatrix}u^k\\v^k\end{pmatrix}.
\end{equation}
Since $(u^0,v^0)=(q^{-1},0)$, induction gives
\begin{equation}
\label{eq:lp-ladmm-third-block-closed-form}
\begin{aligned}
    u^k
    &=
    \frac{\cos^k\theta_K}{q}
    \left[
        \cos(k\theta_K)
        +\tan\theta_K\sin(k\theta_K)
    \right] \\
    &=
    \frac{\cos^{k-1}\theta_K}{q}
    \cos((k-1)\theta_K),\\
    v^k&=u^{k+1}-u^k.
\end{aligned}
\end{equation}
In particular,
\[
    \lvert u^k\rvert
    \leq\frac{1}{q\cos\theta_K}<1,
\]
and therefore
\[
    -1-u^k-v^k=-1-u^{k+1}<0.
\]
Thus $z_3^{k+1}=0$ for every $k$, verifying the branch used above.
Moreover, the matrix in
\eqref{eq:lp-ladmm-third-block-recurrence} has determinant
$\cos^2\theta_K>0$. Since its initial state is nonzero, the third
coordinate cannot reach the optimal state in finitely many
iterations. Hence the ordinary sequence does not terminate finitely.

Since $q\in 8\mathbb N$, $K=q^2$, and
$\theta_K=\pi/(4q)$, we have $K\theta_K\in 2\pi\mathbb N$.
Using \eqref{eq:lp-ladmm-third-block-closed-form} and the finite
geometric-sum formula gives
\[
    \sum_{j=1}^K u^j
    =
    \frac{1-\cos^K\theta_K}{q},
    \qquad
    \sum_{j=1}^K v^j
    =
    -\frac{1-\cos^K\theta_K}{q}.
\]
Together with \eqref{eq:lp-ladmm-first-block-orbit}, this yields
\begin{equation}
\label{eq:lp-ladmm-averages}
\begin{aligned}
    y_a^K
    =\begin{pmatrix}
        0\\
        -\dfrac{1-\cos^K\theta_K}{a_K q K}
      \end{pmatrix}, \quad
    z_a^K
    =\begin{pmatrix}
        \dfrac{K-q}{K}\gamma_K\\
        \dfrac{K-q}{K}\gamma_K\\
        0
      \end{pmatrix},\quad
    x_a^K
    =\begin{pmatrix}
        \dfrac{a_0(q-1)}{2K}\\
        \dfrac{a_0(q-1)}{2K}\\
        1+\dfrac{1-\cos^K\theta_K}{qK}
      \end{pmatrix}.
\end{aligned}
\end{equation}
Substituting \eqref{eq:lp-ladmm-averages} into
\eqref{eq:lp-dual-kkt-residual}, and minimizing over
$\mathcal W_{D,K}^\star$, gives
\begin{equation}
\label{eq:lp-ergodic-residual-distance}
\begin{aligned}
    \norm{\mathcal R_{D,K}(w_a^K)}^2
    &=
    \frac{(q-1)^2}{4K^2}
    +\frac1{K^2}
    +(1+a_K^2)
      \left(\frac{1-\cos^K\theta_K}{qK}\right)^2,\\
    \operatorname{dist}^2(w_a^K,\mathcal W_{D,K}^\star)
    &=
    \frac{(q-1)^2}{4K^2}
    +\frac1{K^2}
    +(1+a_K^{-2})
      \left(\frac{1-\cos^K\theta_K}{qK}\right)^2.
\end{aligned}
\end{equation}
Here the closest point in $\mathcal W_{D,K}^\star$ has
$y_1^\star=0$.

From $\theta_K=\pi/(4q)$ and
$a_K=\sqrt2\sin\theta_K$, we have
\[
    \frac1{\sqrt2q}\leq a_K\leq\sqrt2,
    \qquad
    0
    \leq
    \frac{1-\cos^K\theta_K}{qK}
    \leq
    \frac1{qK}.
\]
Since $K=q^2$ and $q\geq8$,
\[
    \frac{q-1}{2K}
    =
    \frac{1-q^{-1}}{2\sqrt K}
    \geq
    \frac7{16\sqrt K},
\]
which gives the lower bounds in \eqref{eq:lp-ergodic-rates}.
For the upper bounds, \eqref{eq:lp-ergodic-residual-distance} gives
\[
\begin{aligned}
    \norm{\mathcal R_{D,K}(w_a^K)}^2
    &\leq
    \frac1{4K}+\frac1{K^2}+\frac3{K^3}
    \leq\frac1K,\\
    \operatorname{dist}^2(w_a^K,\mathcal W_{D,K}^\star)
    &\leq
    \frac1{4K}+\frac3{K^2}+\frac1{K^3}
    \leq\frac1K.
\end{aligned}
\]
Direct substitution into the primal objective and dual objective
further yields
\[
    p_K(x_a^K)-d_K(y_a^K,z_a^K)
    =
    \frac{q-1}{2qK}
    +\frac{1-\cos^K\theta_K}{qK}.
\]
Thus,
\[
    0
    \leq
    p_K(x_a^K)-d_K(y_a^K,z_a^K)
    \leq
    \frac1{2K}+\frac1{qK}
    \leq
    \frac5{8K}.
\]
This proves \eqref{eq:lp-ergodic-rates}.

We next consider the Halpern-accelerated
Algorithm~\ref{alg:lp-halpern-ladmm}. By symmetry,
$x_1^k=x_2^k$ and $y_1^k=0$. Its first two post-sweep coordinates
satisfy
\[
    \bar x_1^{k+1}
    =
    \bar x_2^{k+1}
    =
    (x_1^k-\gamma_K)_+,
    \qquad
    \bar z_1^{k+1}
    =
    \bar z_2^{k+1}
    =
    (\gamma_K-x_1^k)_+.
\]
The Halpern update therefore gives
\[
    x_1^{k+1}=x_2^{k+1}
    =
    \frac{a_0}{k+2}
    +
    \frac{k+1}{k+2}(x_1^k-\gamma_K)_+.
\]
A direct induction shows that
\[
    x_1^k=x_2^k=
    \begin{cases}
        a_0-k\gamma_K/2, & 0\leq k\leq2q-2,\\
        a_0/(k+1),       & k\geq2q-1.
    \end{cases}
\]
Since $K-1\geq2q-1$, the $K$th post-sweep point satisfies
\[
    \bar x_1^K=\bar x_2^K=0,
    \qquad
    \bar z_1^K=\bar z_2^K
    =
    \gamma_K-\frac{a_0}{K}.
\]

For the third coordinate, let
\[
    \xi^k:=(u_H^k,v_H^k),
    \qquad
    u_H^k:=x_3^k-1,
    \qquad
    v_H^k:=a_Ky_2^k,
\]
denote the Halpern anchor state, and let $M_K$ be the matrix in
\eqref{eq:lp-ladmm-third-block-recurrence}. We prove simultaneously
by induction that the zero branch is active and that
\[
    \xi^k
    =
    \frac1{k+1}\sum_{j=0}^kM_K^j\xi^0.
\]
The displayed identity is immediate for $k=0$. Suppose that it holds
at some $k$. The candidate post-sweep state on the zero branch is
$M_K\xi^k$, whose first component is
\[
    \frac1{k+1}\sum_{j=1}^{k+1}u^j.
\]
By \eqref{eq:lp-ladmm-third-block-closed-form}, every term in this
average is strictly larger than $-1$. Hence
\[
    \bar z_3^{k+1}
    =
    \max\{-1-(M_K\xi^k)_1,0\}
    =
    0,
\]
so this candidate is indeed the actual post-sweep state. The Halpern
update then gives
\[
    \xi^{k+1}
    =
    \frac1{k+2}\xi^0
    +
    \frac{k+1}{k+2}M_K\xi^k
    =
    \frac1{k+2}\sum_{j=0}^{k+1}M_K^j\xi^0,
\]
which closes the induction. Consequently, the $K$th post-sweep state is
\[
    M_K\xi^{K-1}
    =
    \frac1K\sum_{j=1}^KM_K^j\xi^0
    =
    \frac1K\sum_{j=1}^K(u^j,v^j)
    =
    \frac{1-\cos^K\theta_K}{qK}(1,-1).
\]
Thus,
\[
    \bar x_3^K-1
    =\frac{1-\cos^K\theta_K}{qK},
    \qquad
    a_K\bar y_2^K
    =-\frac{1-\cos^K\theta_K}{qK},
    \qquad
    \bar z_3^K=0.
\]
Consequently,
\begin{equation}
\label{eq:lp-halpern-output}
\begin{aligned}
    \bar y^K
    =\begin{pmatrix}
        0\\
        -\dfrac{1-\cos^K\theta_K}{a_K q K}
      \end{pmatrix}, \quad
    \bar z^K
    =\begin{pmatrix}
        \gamma_K-a_0/K\\
        \gamma_K-a_0/K\\
        0
      \end{pmatrix}, \quad
    \bar x^K
    =\begin{pmatrix}
        0\\
        0\\
        1+\dfrac{1-\cos^K\theta_K}{qK}
      \end{pmatrix}.
\end{aligned}
\end{equation}
Substitution into \eqref{eq:lp-dual-kkt-residual} and minimization over
$\mathcal W_{D,K}^\star$ give
\begin{equation}
\label{eq:lp-halpern-residual-distance}
\begin{aligned}
    \norm{\mathcal R_{D,K}(\bar w^K)}^2
    &=
    \frac1{K^2}
    +(1+a_K^2)
      \left(\frac{1-\cos^K\theta_K}{qK}\right)^2,\\
    \operatorname{dist}^2(\bar w^K,\mathcal W_{D,K}^\star)
    &=
    \frac1{K^2}
    +(1+a_K^{-2})
      \left(\frac{1-\cos^K\theta_K}{qK}\right)^2.
\end{aligned}
\end{equation}
The term $K^{-2}$ gives the lower bounds in
\eqref{eq:lp-halpern-rate}, while the bounds established above give
\[
\begin{aligned}
    \norm{\mathcal R_{D,K}(\bar w^K)}^2
    &\leq
    \frac{1+3/K}{K^2},\\
    \operatorname{dist}^2(\bar w^K,\mathcal W_{D,K}^\star)
    &\leq
    \frac{3+1/K}{K^2}.
\end{aligned}
\]
Finally, \eqref{eq:lp-halpern-output} gives
\[
    p_K(\bar x^K)-d_K(\bar y^K,\bar z^K)
    =
    \frac{1-\cos^K\theta_K}{qK},
\]
and hence
\[
    0
    \leq
    p_K(\bar x^K)-d_K(\bar y^K,\bar z^K)
    \leq
    \frac1{qK}
    =
    \frac1{K^{3/2}}.
\]
This proves \eqref{eq:lp-halpern-rate} and completes the proof.
\end{proof}

This proposition shows that $O(K^{-1})$ ergodic primal feasibility,
dual feasibility, and primal--dual gap do not imply the same rate for
the distance to the KKT solution set. In contrast, under an error
bound, the KKT residual is comparable to this distance.
Furthermore, the KKT residual and solution-set distance of the equal-weight ergodic average are both $\Theta(K^{-1/2})$, whereas Halpern acceleration achieves $\Theta(K^{-1})$ for both.

Although Applegate et al.\ \cite{ApplegateHinderLuLubin2023}
established $O(k^{-1})$ ergodic rates for PDHG, a linearized ADMM
method \cite{esser2010general,chambolle2011first}, in terms of primal
infeasibility, dual infeasibility, and the primal--dual gap, these
bounds yield only an $O(k^{-1/2})$ rate for the KKT residual
\eqref{eq:lp-dual-kkt-residual}. Similarly, the $O(k^{-1})$ ergodic
rates of Chambolle and Pock
\cite{chambolle2011first,chambolle2016ergodic} for primal--dual
gap-type measures yield only an $O(k^{-1/2})$ KKT residual rate.
Our example shows that this order is attained at the prescribed
horizon. In contrast, HPR-LP \cite{chen2026hprlp} directly achieves
an $O(k^{-1})$ KKT residual complexity bound for LP.

\subsection{Solver developments}

Beyond these complexity guarantees, accelerated ADMM-type methods have
also led to efficient large-scale solvers. Specifically, HPR-LP
\cite{chen2026hprlp} has demonstrated strong performance on large-scale
LP instances and outperformed GPU implementations of PDLP based on
ergodic PDHG on several benchmark sets
\cite{applegate2021practical,ApplegateHinderLuLubin2023,lu2025cupdlp}.
The HPR-LP framework further led to the reflected restarted
Halpern--PDHG method (r2HPDHG) \cite{lu2024restarted}, implemented in
cuPDLPx \cite{lu2025cupdlpx}, which is a special case of the HPR
method
\cite{SunYuanZhangZhao2025,chen2026hprlp,ChenSunYuanZhangZhao2025}.
Moreover, the HPR framework has shown strong empirical performance
beyond LP, including optimal transport through HOT~\cite{zhang2025hot}
and convex quadratic composite programming through
HPR-QP~\cite{chen2025hprqp}.

\section{Conclusion}
We constructed a fixed-dimensional, horizon-dependent family for
which, at each prescribed square horizon $K=q^2$, the KKT residual of
classical ADMM is $\Omega(K^{-1/2})$ at both the last iterate and the
equal-weight ergodic average. Since these horizons are unbounded,
neither output admits a uniform $O(K^{-1})$ KKT residual bound for
classical ADMM. We further constructed an LP example showing that $O(K^{-1})$ primal feasibility, dual feasibility, and primal--dual gap do not imply the same rate for the distance to the KKT solution set, whereas under an error-bound condition the KKT residual is comparable to this distance up to constant factors. The same example also demonstrates that
Halpern acceleration can improve both the KKT residual and the
solution-set distance from $\Theta(K^{-1/2})$ for the equal-weight
ergodic average to $\Theta(K^{-1})$.

\section*{Acknowledgments}
GPT-5.6 was used as an auxiliary tool in developing the lower-bound
construction. All mathematical arguments were independently verified
by the authors. The work of Defeng Sun was supported by the Research Center for Intelligent Operations Research, the RGC Senior Research Fellow Scheme (No. SRFS2223-5S02), and the RGC General Research Fund (Project No. 15307822). The work of Yancheng Yuan was supported by the RGC Early Career Scheme (Project No. 25305424), the NSFC Young Scientists Fund (Project No. 12501440), and the Research Center for Intelligent Operations Research. The work of Xinyuan Zhao was supported in part by the National Natural Science Foundation of China under Project No. 12271015.

\bibliographystyle{plain}
\bibliography{reference}

\end{document}